\documentclass[11pt,a4paper]{article}
\usepackage[english]{babel}
\usepackage{a4wide,epsfig,graphics,float,subfigure}
\usepackage{graphicx,color}
\usepackage{amsmath}
\usepackage{amssymb}
\usepackage{datetime}
\usepackage{multirow}
\usepackage{slashbox}
\usepackage{boxedminipage}
\usepackage{fancybox}       
\usepackage{fancyhdr}       
\title{Error estimates for Stokes problem with Tresca friction condition}
\newtheorem{theorem}{Theorem}[section]
    \newtheorem{lemma}[theorem]{Lemma}
    \newtheorem{proposition}[theorem]{Proposition}

\newenvironment{proof}{\begin{trivlist}
                       \item[]\hspace{0cm}{\bf Proof : }
                       \hspace{0cm} }{\hfill $\square$
                     \end{trivlist}}

    \newcommand{\qed}{\nobreak \ifvmode \relax \else
          \ifdim\lastskip<1.5em \hskip-\lastskip
          \hskip1.5em plus0em minus0.5em \fi \nobreak
          \vrule height0.75em width0.5em depth0.25em\fi}

\begin{document}
 \thispagestyle{empty}

\vspace*{3cm}
\centerline{\huge Error estimates for Stokes problem with Tresca}

\vspace{0.5cm}
\centerline{\huge friction condition}


%
%

\vspace*{5cm}
\begin{center}

{\bf Mekki AYADI}$^1$  {\bf Mohamed Khaled GDOURA}$^{1,2}$ {\bf Taoufik SASSI}$^2$ 
\end{center}
{\footnotesize
\noindent$^1$ Laboratoire de Mod\'elisation Math\'ematiques et Num\'erique dans les Sciences de l'Ing\'enieur, Ecole Nationale d'Ing\'enieurs de Tunis, B.P. 32, 1002 Tunis, Tunisie.

\noindent$^2$ Laboratoire de math\'ematiques Nicolas Oresme, CNRS UMR 6139, Universit\'e de Caen, UFR sciences Campus II, Bd Mar\'echal JUIN, 14032 Caen cedex, France.

\noindent Emails: mekki.ayadi@enis.rnu.tn, mohamedkhaled.gdoura@lamsin.rnu.tn (corresponding author), taoufik.sassi@unicaen.fr.
%

\vspace{0.25cm}
%
\section*{R\'esum\'e}
Dans ce travail on a propos\'e et \'etudi\'e une formulation mixte \`a trois champs pour r\'esoudre le probl\`eme de Stokes avec des conditions aux limites non-lin\'eaires, du type Tresca. Deux multiplicateurs de Lagrange ont \'et\'e utilis\'es afin d'imposer $div(u)=0$ et de r\'egulariser la fonctionnelle \'energie. Les \'el\'ements finis {\it P1 bulle/P1-P1} ont permis de discr\'etiser le probl\`eme r\'esultant. Des estimations d'erreurs ont \'et\'e d\'eriv\'ees et plusieurs tests num\'eriques sont r\'ealis\'es.

\vspace{0.25cm}
\noindent{\it {\bf Mots cl\'es}: Probl\`eme de Stokes, Frottement de Tresca, in\'equation variationnelle, \'elements finis mixtes, estimation d'erreur.}
%
\section*{Abstract}
In this work we propose and study a three field mixed formulation for solving the Stokes problem with Tresca-type non-linear boundary conditions. Two Lagrange multipliers are used to enforce $div(u)=0$ constraint and to regularize the energy functional. The resulting problem is discretised using {\it P1 bubble/P1-P1} finite elements. Error estimates are derived and several numerical studies are achieved.

\vspace{0.25cm}
\noindent{\it {\bf Key words}: Stokes problem, Tresca friction, variational inequality, mixed finite element, error estimates.}
\newpage
%
\section*{Notations}
We need to set some notations and recall some functional tools necessary for our analysis. Let $\Omega\subset \mathbb{R}^d$, $d\ge2$, an open set with boundary $\partial \Omega$ wich is the union of  two nonoverlapping portions $\Gamma_0 $ and $\Gamma $ (may be empty).

\noindent The euclidian norm of a point ${\bf x}\in \mathbb{R}^d$ is denoted by $|{\bf x}|$ in what follows.
The Lebesgue space $L^2(\Omega)$ is endowed with the norm:
$$
\displaystyle \forall p\in L^2(\Omega) \qquad ||p||_0= \left( \int_\Omega|p({\bf x})|^2\, d{\bf x}\right)^\frac{1}{2},
$$
while $L^2_0(\Omega)$ is the closed subspace of $L^2(\Omega)$ defined by:
$$
\displaystyle L^2_0(\Omega)=\left\lbrace  p\in  L^2(\Omega) \mbox{ such that }  \int_\Omega p({\bf x})\, d{\bf x} =0\right\rbrace.
$$

\noindent We make constant use of the standard Soblev space $H^m(\Omega)$, $m\ge1$, provided with the norm:
$$
\displaystyle||\psi||_m	=\left( \sum\limits_{0\leq |\alpha| \leq m}||\partial^\alpha \psi||^2_0\right) ^{1/2},
$$
where $\alpha$ is a multi-index. Fractional Sobolev spaces $H^\nu(\Omega)$, $\nu\in \mathbb{R}_+\backslash \mathbb{N}$ are defined by
$$
\displaystyle H^\nu(\Omega)= \left\lbrace \varphi \in H^m(\Omega) \mbox{ such that } ||\varphi||_{\nu,\Omega}<+\infty \right\rbrace,
$$
with
$$
\displaystyle ||\varphi||_{\nu ,\Omega}=\left( ||\varphi||^2_{m}+\sum\limits_{|\alpha|=m}\int_\Omega\int_\Omega \frac{(\partial^{\alpha}\varphi(\textbf{x})-\partial^{\alpha}\varphi(\textbf{y}))^{2}}{|\textbf{x}-\textbf{y}|^{2+2\theta}}\, d\textbf{x}\, d\textbf{y}\right)
$$
with $m$ being the integr part of $\nu$ and $\theta$ its decimal parts.

\noindent The closure in $H^\nu(\Omega)$ of $D(\Omega)$, the space of infinitely differentiable functions with support in $\Omega$, is denoted by $H^\nu_0(\Omega)$. On any portion $\Gamma \subseteq \partial \Omega$ we introduce the space $H^\frac{1}{2}(\Gamma)$ as follows
$$
\displaystyle H^{\frac{1}{2}}(\Gamma)=\left\lbrace \varphi \in L^2(\Gamma)\mbox{ such that } ||\varphi||_{\frac{1}{2},\Gamma}<+\infty\right\rbrace,
$$
where
$$
\displaystyle ||\psi||_{\frac{1}{2},\Gamma}=\left( ||\psi||_{0,\Gamma}+\int_\Gamma\int_\Gamma \frac{(\psi(\textbf{x})-\psi(\textbf{y}))^2}{|\textbf{x}-\textbf{y}|^2}d\Gamma_\textbf{x} d\Gamma_\textbf{y}\right)^\frac{1}{2}.
$$

The space $H^{-\frac{1}{2}}(\Gamma)$ is the dual space of $H^{\frac{1}{2}}(\Gamma)$, $\left\langle \cdot ,\cdot\right\rangle $ stands for the duality pairing and 
$$
\displaystyle ||\mu||_{-\frac{1}{2},\Gamma}=\sup\limits_{\varphi\in {H}^{\frac{1}{2}}(\Gamma), \varphi\neq 0}\frac{\left\langle \mu,\varphi\right\rangle }{||\varphi||_{\frac{1}{2},\Gamma}}.
$$

The special space ${H}_{00}^{\frac{1}{2}}(\Gamma)$ is defined as the set of the restriction to $\Gamma$ of the functions of $H^{\frac{1}{2}}(\partial \Omega)$ that vanish on $\partial \Omega\backslash \Gamma$ and its dual space is denoted by $({H}^{\frac{1}{2}}_{00}(\Gamma))'$.

The cartesian product of $k$ previous spaces and their elements are denoted by bold caracter. The respective norms are introduced as follows:
\begin{eqnarray*}
\begin{array}{ll}
 \displaystyle ||{\bf v}||_m	&=\displaystyle\left( \sum\limits_{i=0}^k||v_i||^2_k\right) ^{1/2} \quad {\bf v}=(v_1,\cdots,v_m)\in{\bf H}^1(\Omega),\\\\
\displaystyle ||{\bf w}||_{\frac{1}{2},\Gamma}&=\displaystyle\left( \sum\limits_{i=0}^k||{\bf w}_i||^2_{\frac{1}{2} ,\Gamma}\right) ^{1/2} \quad {\bf {\bf w}}=(w_1,\cdots,w_m)\in{\bf H}^\frac{1}{2}(\Gamma),\\\\
\displaystyle ||{\bf \mu}||_{-\frac{1}{2},\Gamma}&=\displaystyle\left( \sum\limits_{i=0}^k||{\bf \mu}_i||^2_{-\frac{1}{2} ,\Gamma}\right) ^{1/2} \quad {\bf {\bf \mu}}=(\mu_1,\cdots,\mu_m)\in{\bf H}^{-\frac{1}{2}}(\Gamma),\\\\
\end{array}
\end{eqnarray*}

Let $\mathcal{X}\subset H^1(\Omega)$  be a subspace of functions vanishing on a non-empty portion $\Gamma_0$ open in $\partial \Omega$
$$
\mathcal{X}=\left\lbrace v\in H^1(\Omega) \mbox{ such that } v_{|_{\Gamma_0}}=0\right\rbrace .
$$
We introduce the enrgetic norm $|||\cdot |||_1$ in $\mathcal{X}$ corresponding to the scalar product
$$
\displaystyle\left(\textbf{u}, \textbf{v}\right)_1=\int_\Omega\sum\limits_{i,j=1}^3 \varepsilon_{ij}(\textbf{u})\varepsilon_{ij}(\textbf{v})d(\textbf{x}),
$$
where $\varepsilon_{ij}$ is the $ij$-th component of the linearized strain rate tensor $\displaystyle \varepsilon({\bf u})=\frac{1}{2}(\nabla {\bf u}+\nabla^t {\bf u})$. From the Korn inequality it follows that $||\cdot||_1$ and $|||\cdot |||_1$ are equivalent in $\mathcal{X}$.

We denote by ${\bf n}$ the outward unit normal to $\partial \Omega$ and ${\bf u}_n$, respectively ${\bf u}_t$, the normal , respectively the tangential, component of ${\bf u}$.

The stress vector $\sigma$ is equal to $\underline{\underline{\sigma}}.{\bf n}$ where $\underline{\underline{\sigma}}$ is the Cauchy stress tensor defined by:
$$
\underline{\underline{\sigma}}=2\nu\varepsilon({\bf u})-p \underline{\underline{\delta}},
$$
where $p$ is the hydrostatic pressure, $\underline{\underline{\delta}}$ is the identity tensor and $\nu$ is the kinematic fluid viscosity.
\section{Introduction}
No-slip hypothesis at fluid-wall interface leads to good agreement with experimental observations for newtonian fluids which is no longer true for non-newtonian fluid \cite{magnin87}. For example, in the flow of certain high molecular weight linear polymers through circular dies, the exit flow rate has been found to be a discontinous function of pressure drop over a certain range of shear rates \cite{tordell86,hatzik91}. This obervation is consistant with the hypothesis that the velocity at the wall is not zero. Several studies have been made and showed not only that slip takes place when a threshold is reached \cite{hatzik93} but also it's the origin of many defects and instabilities in the polymer injection process \cite{santana05,goutille}.

The first attempt to integrate this boundary condition in a numerical simulation of a flow is due to Doltsini {\it et al.} \cite{Doltsini87} and Fortin \cite{fortin91}. Since that, many papers were published simulating various flows with such boundary conditions (see \cite{rao99} and refrences therin). Recently, based on the penality method, error estimates for the Stokes problem with Tresca boundary conditions with strong regularity assumption on the velocity field are obtained \cite{kaitai08}.

The aim of this work is to contribute to the numerical analysis of Stokes problem with Tresca boundary conditions. Our first purose is to carry out the convergence analysis and a priori estimates for the mixed finite element formulation of the above cited problem. The second one is to derive an algorithm well adapted to this formulation and easy to implement in order to validate our theoritical estimates. 

The paper is organized as follows. First, we introduice the equations modelling the Stokes problem. Then we establish the continous mixed variational formulation is section 3. The following section is devoted to a priori error estimates , we show an optimal order of $h^{3/4}$ with ${\bf H}^2(\Omega)$ assupmtion regularity on the velocity. In section 5 we propose an algorithm based on augmented lagrangian method to solve the 2D problem and make some numerical tests.
\section{Setting Stokes problem with nonlinear boundary conditions}
We consider the following Stokes problem with nonlinear boundary condition of Tresca friction type:
\begin{eqnarray}
\label{Stok_tres_cont}
\left\{
\begin{array}{rcllrr}
-div(\nu\varepsilon( {\bf u})) + \nabla p	&=&\textbf{f}			& \mbox{in }\Omega\\
div({\bf u})					&=&0 				& \mbox{in }\Omega \\
{\bf u}						&=&\textbf{0}			& \mbox{on }\Gamma_0\\
{\bf u}_n					&=&\textbf{0}			& \mbox{on }\Gamma\\
|\sigma_t|< g \Rightarrow  \quad {\bf u}_t	&=&\textbf{0}			&\mbox{on }\Gamma\\
|\sigma_t|= g \Rightarrow \exists k>0 \mbox{ a constant such that }\quad {\bf u}_t	&=&-k \sigma_t			&\mbox{on }\Gamma
\end{array}
\right.
\end{eqnarray}
with $\Omega \subset \mathbb{R}^d$ ($d=2$ or $3$) an open set with regular boundary $\partial \Omega$, which is the union of two nonoverlapping portions $\Gamma_0 $ and $\Gamma $. $\Gamma_0 $ is subjected to no-slip boundary condition while $\Gamma$ is where le fluid may slip.
We need this result to derive the variational problem.
%
\begin{proposition}\cite{saidi04}
\begin{equation}
\label{equiv_tresca}
\left\{
\begin{array}{rcllrr}
|\sigma_t|<g \, \Rightarrow {\bf u}_t					&=&0\\
|\sigma_t|=g \, \Rightarrow {\bf u}_t					&=&-k \sigma_t \quad \mbox{ k is a non-negative constant on }\Gamma
\end{array}
\right.
\qquad \Longleftrightarrow \qquad 
\sigma_t.{\bf u}_t + g|{\bf u}_t|=0 \mbox{ on }\Gamma
\end{equation}
\end{proposition}
One can derive the variational formulation of (\ref{Stok_tres_cont}):
\begin{eqnarray}
\label{ineq_var_stokes_gen_sigm}
\left\{
\begin{array}{l}
\mbox{Find } {\bf u} \in {\bf V}_{div}(\Omega) \mbox{ such that :}\, \forall \, {\bf v}\in {\bf V}_{div}(\Omega)\\\\
\displaystyle a({\bf u},{\bf v}-{\bf u})+j({\bf v})-j({\bf u})\geq L({\bf v}-{\bf u}),
\end{array}
\right.
\end{eqnarray}
with 
$$
 {\bf V}(\Omega)=\{ {\bf v}\in {\bf H}^1(\Omega) ,\; {\bf v}_{|\Gamma_0}=0 ,\\ {\bf v}.{\bf n}_{|\Gamma}=0\} \qquad\quad {\bf V}_{div}(\Omega)=\{ {\bf v}\in {\bf V}(\Omega)\ ,\; div({\bf v})=0 \mbox{ in }\Omega\}
$$
$$\displaystyle a({\bf u},{\bf v})= \int_{\Omega}\nu \varepsilon({\bf u})\colon \varepsilon({\bf v})\, d\Omega \qquad \displaystyle L({\bf v})=\int_{\Omega} {\bf f}\, {\bf v}\, d\Omega \qquad g \mbox{ a non-negative function in } L^2(\Gamma)$$
$$\qquad \mbox{ and } \qquad\displaystyle j({\bf v})=\int_{\Gamma}g|{\bf v}_t| \,d\Gamma \qquad \forall{\bf u}, {\bf v} \in {\bf V}_{div}(\Omega).$$

\noindent Problem (\ref{ineq_var_stokes_gen_sigm}) is an elliptic variationnal inequality of the second kind which has a unique solution \cite{glow}. Moreover, since the bilinear form $a(\cdot,\cdot)$ is symmetic (\ref{ineq_var_stokes_gen_sigm}) is equivalent to the following constrained non-differentiable minimization problem:

\begin{eqnarray}
\label{pb_minimi}
\left\{
\begin{array}{l}
\mbox{Find } {\bf u} \in {\bf V}_{div}(\Omega) \mbox{ such that :}\\\\
\displaystyle \mathcal{J}({\bf u}) \leq \mathcal{J}({\bf v}) \quad \forall \, {\bf v}\in {\bf V}_{div}(\Omega),
\end{array}
\right.
\end{eqnarray}
where $\displaystyle \mathcal{J}({\bf v})=\frac{1}{2} \, a({\bf v}, {\bf v}) + j({\bf v}) -L({\bf v})$.

\section{Mixed Formulation}
\noindent In order to solve (\ref{pb_minimi}) a Lagrange multiplier $q$ is needed to enforce the condition $div({\bf u})=0$ in $\Omega$, which can be identified with the pressure. In the other hand Fujita proved in \cite{fujita_coh} that (\ref{ineq_var_stokes_gen_sigm}) is equivalent to 

\begin{eqnarray}
\label{ineq_var_sigma}
\left\{
\begin{array}{l}
\exists \, \sigma_t \in ({\bf H}^{\frac{1}{2}}_{00}(\Gamma))' \mbox{ such that }|\sigma_t|\leq g \mbox{ on }\Gamma\\
\displaystyle \int_{\Gamma} \sigma_t ({\bf v}-{\bf u})_t + j({\bf v})- j({\bf u}) \geq 0 \quad \forall \, {\bf v} \in {\bf V}(\Omega).
\end{array}
\right.
\end{eqnarray}
$\sigma_t$ is seen as a Lagrange multiplier and can be identified with the shear stress on $\Gamma$. The minimization problem (\ref{pb_minimi}) is equivalent to the following saddle-point formulation :
%
\begin{eqnarray}
\label{pt_selle}
\left\{
\begin{array}{l}
\mbox{Find }({\bf u},p, \lambda) \in \mathcal{H} \quad \mbox{such that:}\\\\
\mathcal{L}({\bf u},q,\mu) \leq \mathcal{L}({\bf u},p,\lambda) \leq \mathcal{L}({\bf v},p,\lambda)\quad
\forall ({\bf v},q,\mu)\in \mathcal{H}. 
\end{array}
\right.
\end{eqnarray}
%

\begin{eqnarray*}
\label{lagrang}
\begin{array}{l}
\displaystyle \mathcal{L}({\bf v},q,\mu)=\frac{1}{2} \, a({\bf v},{\bf v})-\int_{\Omega}q \,  div({\bf v}) + \int_{\Gamma}\mu \, {\bf v}_t -L({\bf v}) \quad
\forall ({\bf v},q,\mu)\in \mathcal{H}={\bf V}\times L^2_0(\Omega)\times \mathcal{Q}
\end{array}
\end{eqnarray*}
where 
\begin{eqnarray*}
\label{covx}
\begin{array}{rl}
\displaystyle \mathcal{Q}	&=	\left\lbrace \displaystyle\mu\in (L^2(\Gamma))^{d-1}, \, |\mu|\leq g \right\rbrace \\\\
\displaystyle			&=	\left\lbrace \displaystyle\mu\in (L^2(\Gamma))^{d-1}, \,\int_{\Gamma}\mu {\bf v}_t - \int_{\Gamma} g |{\bf v}_t|\leq 0\, \forall {\bf v}\in {\bf V}\right\rbrace .
\end{array}
\end{eqnarray*}

\noindent According to \cite{kikuchi88} problem (\ref{pt_selle}) has a uniqe solution charcterized by
%
%
\begin{eqnarray}
\label{pb_mix_3cham_vect}
\left\{
\begin{array}{rcll}
\mbox{Find }({\bf u},(p, \lambda)) \in {\bf V}\times \Lambda \quad \mbox{such that:}\\\\
a({\bf u},{\bf v})+ b((p,\lambda),{\bf v})	&=&	L({\bf v})	& \forall {\bf v}\in {\bf V}\\\\
\displaystyle 		b((q-p,\mu-\lambda),{\bf u})	&\leq&	0		& \forall (q,\mu) \in \Lambda,
\end{array}
\right.
\end{eqnarray}
where
\begin{equation}
b((p,\lambda),{\bf v})=-(p,div {\bf v}) + \langle \lambda, {\bf v}_{t}\rangle
\end{equation}
and $\Lambda=L^2_0(\Omega)\times \mathcal{Q}$ is a closed convex of $\mathcal{M}=L^2_0(\Omega)\times (L^2(\Gamma))^{d-1}$.
%

\begin{lemma}
 There exists a constant $ \alpha >0$ such that : $\forall (q,\mu)\in \mathcal{M}$
\begin{equation}
\label{infsup_cont}
\displaystyle \sup\limits_{\bf v\in V}\frac{b((q,\mu),{\bf v})}{||{\bf v}||_{1,\Omega} }\geq\alpha\left( ||q||+||\mu||_{-\frac{1}{2}} \right).
\end{equation}
\end{lemma}
\begin{proof}
To prove this result we are inspired by \cite{kikuchi88}. We have to prove that for all $(q,\mu)\in \mathcal{M}$ there exists ${\bf u}\in {\bf V}$ such that:
 \begin{eqnarray}
\left\{
\begin{array}{ccll}
\label{pb_infsup_cont}
\displaystyle div\,{\bf u}	&=&	q			&\mbox{in }\Omega\\
 {\bf u}			&=&	\textbf{0}		&\mbox{on }\Gamma_0\\
{\bf u}_n			&=&	\textbf{0}		&\mbox{on }\Gamma\\
{\bf u}_t			&=&	h^{-1}(\mu)		&\mbox{on }\Gamma
\end{array}
\right.
\end{eqnarray}
wich satisfies
\begin{equation}
\label{stab_pmv}
\displaystyle ||{\bf u}||_1 \leq C \, \left( ||q||_0 + ||\mu||_{-\frac{1}{2}} \right),
\end{equation}
where $h^{-1}(\cdot)$ is the inverse Riesz operator $({\bf H}^{\frac{1}{2}}_{00}(\Gamma))' 	\rightarrow {\bf H}^{\frac{1}{2}}_{00}(\Gamma)$.

The proof is diveded into five steps:

\noindent {\bf Step 1}

We suppose that $\Omega$ is a convex with a regular boudary $\partial \Omega$. Let $q\in L^2_0(\Omega)$ and $\phi_1$ be the solution of:
\begin{eqnarray}
\label{step1}
\left\{
\begin{array}{llll}
\Delta \phi_1	&=&	q	&\mbox{ on }\Omega\\
\phi_1		&=&	0	&\mbox{ in }\partial\Omega
\end{array}
\right.
\end{eqnarray}
According to \cite{grisvard75} problem (\ref{step1}) admits a unique solution $\phi_1$ verifing $\phi_1\in H^2(\Omega)$ and 
\begin{equation}
\label{stab_fi1}
 \displaystyle ||\phi_1||_2\leq C ||q||_0
\end{equation}
where $C$ is a constant independent of and $q$.

\noindent {\bf Step 2}

\noindent Since $\phi_1\in H^2(\Omega)$, $\displaystyle{\frac{\partial \phi_1}{\partial n}}\in H^{\frac{1}{2}}(\partial\Omega)$, we now consider the following Neumann problem:
\begin{eqnarray}
\left\{
\begin{array}{llll}
\Delta \phi_2						&=&	0					&\mbox{ on }\Omega\\
\displaystyle\frac{\partial\phi_2}{\partial n}		&=&	\displaystyle-\frac{\partial \phi_1}{\partial n}	&\mbox{ in }\partial\Omega
\end{array}
\right.
\end{eqnarray}
According to \cite{grisvard75}, this problem admits a regular solution $\phi_2\in H^2(\Omega)$ verifing:
\begin{equation}
\label{stab_fi2}
 \displaystyle ||\phi_2||_2\leq C ||\frac{\partial \phi_1}{\partial n}||_{\frac{1}{2}}.
\end{equation}
\noindent {\bf Step 3}

Let $\psi$ be the unique solution to the following bilaplacian problem :
\begin{eqnarray}
\label{bilapl}
\left\{
\begin{array}{rlll}
\displaystyle \Delta^2 \psi			&=&	0		&\mbox{ in }\Omega\\\\
\displaystyle \psi				&=&	0		&\mbox{ in }\partial\Omega\\\\
\displaystyle \frac{\partial \psi}{\partial n}	&=&	\chi		&\mbox{ in }\partial\Omega
\end{array}
\right.
\end{eqnarray}
where 
\begin{eqnarray*}
\displaystyle\chi=
\left\{
\begin{array}{llll}
\displaystyle-\frac{\partial \phi_1}{\partial \tau}-\frac{\partial \phi_2}{\partial \tau}			&\mbox{ on }\Gamma_0\\\\
\displaystyle-\frac{\partial \phi_1}{\partial \tau}-\frac{\partial \phi_2}{\partial \tau}+h^{-1}(\mu)		&\mbox{ on }\Gamma,
\end{array}
\right.
\end{eqnarray*}
we can easily show that $\chi\in H^{\frac{1}{2}}(\partial \Omega)$. From \cite{girault79} it holds

\begin{equation}
\label{stab_psi}
 \displaystyle \psi \in H^2(\Omega) \qquad \qquad \qquad ||\psi||_2\leq C||\chi||_{\frac{1}{2}}.
\end{equation}

\noindent {\bf Step 4}

Setting ${\bf u}=\nabla \phi_1 + \nabla \phi_2+curl\,\psi$ with
\begin{eqnarray}
curl\, \psi=
\left[
\begin{array}{c}
\displaystyle \frac{\partial \psi}{\partial x_2}\\\\
\displaystyle -\frac{\partial \psi}{\partial x_1}\\
\end{array}
\right]
\end{eqnarray}
so that ${\bf u}_t=h^{-1}(\mu)$ on $\Gamma$.

\noindent Furthermore we obtain:
\begin{eqnarray}
\label{conti_u_1}
\begin{array}{rcl}
||{\bf u}||_1	&\leq&	\displaystyle||\nabla \phi_1||_1 + ||\nabla \phi_2||_1 + ||curl \psi||_1	\\\\
		&\leq&	\displaystyle||\phi_1||_2 + ||\phi_2||_2 + ||\psi||_2	\\\\
		&\leq&	\displaystyle C\left( ||q||_0 + ||\frac{\partial \phi_1}{\partial n}||_{\frac{1}{2}} + ||\chi||_{\frac{1}{2}}\right) .	\\\\
\end{array}
\end{eqnarray}
where $C>0$ is a generic constant.

\noindent Using inequalities (\ref{stab_fi1}), (\ref{stab_fi2}), (\ref{stab_psi}), the continuity of normal trace application from ${\bf H}^1(\Omega)$ onto ${\bf H}^\frac{1}{2}(\partial\Omega)$ and the continuity of $h^{-1}$, (\ref{conti_u_1}) becomes:
\begin{equation*}
\displaystyle ||{\bf u}||_1 \leq C \, \left( ||q||_0 + ||\mu||_{-\frac{1}{2}} \right).
\end{equation*}

\noindent {\bf Step 5}

\begin{eqnarray}
\label{inf_sup}
\begin{array}{rcl}
\displaystyle \sup\limits_{\bf v\in V}\frac{b((q,\mu),{\bf v})}{||{\bf v}||_{1,\Omega} }	&\geq&	\displaystyle \frac{b((q,\mu),{\bf u})}{||{\bf u}||_{1} },\\\\
 	&\geq&		\displaystyle\frac{||q||^2+||\mu||^2_{-\frac{1}{2}}}{||{\bf u}||_{1}},\\\\
 	&\geq&		\displaystyle\frac{1}{C}\frac{||q||^2+||\mu||^2_{-\frac{1}{2}}}{||q||+||\mu||_{-\frac{1}{2}}},\\\\
	&\geq&		\displaystyle\frac{1}{2C} \left( ||q||+||\mu||_{-\frac{1}{2}} \right).
\end{array}
\end{eqnarray}
Then take $\displaystyle\alpha=\frac{1}{2C}$ to finish the proof.
\end{proof}
%
\begin{theorem}\cite{kikuchi88}
 \label{exist_uniq_mix_vect}
Suppose that $a(\cdot,\cdot)$ is continuous, V-elliptic bilinear form on ${\bf V}(\Omega)$ and (\ref{infsup_cont}) holds. Then there exists a unique $({\bf u},(p, \lambda))$ solution of mixed problem (\ref{pb_mix_3cham_vect}). Moreover, $({\bf u},(p, \lambda))$ is also the unique solution of the saddle-point problem (\ref{pt_selle}). 
\end{theorem}
%
%
\section{Error estimates}
The present section is devoted to finite element approximation of the saddle-point problem (\ref{pt_selle}). The key point lies in finite element discretization of the closed convex $\mathcal{Q}$ of the Lagrange multipliers which leads to a well-posed discrete problem and gives a good convergence rate for the approximate solution.

\noindent We use classical {\it P1 bubble-P1} finite element to disretize $({\bf u}, p)$ and {\it P1} finite element on $\Gamma$ for the Lagrange multiplier $\lambda$. This choice is motivated by Brezzi's and Sassi's results, see \cite{brezzi78,Sassi_baill}.

\noindent $\Omega$ is supposed to be polygonal. Let $\mathcal{T}_h$ be a regular partition of $\overline{\Omega}$ with triangles in the sense of \cite{ciarlet80}. We denote by $\mathbb{P}_n(\kappa)$ the space of polynomials of degree less and equal to $n\in \mathbb{N}$ defined on $\kappa\in \mathcal{T}_h$. We denote by $\mathcal{B}_{\kappa}$ the space of bubble functions defined on $\kappa$ which is a sub-space of $H^1_0(\kappa)$. Then we can define the following discrete spaces :

%
\begin{eqnarray*}
 \begin{array}{l}
\displaystyle \mathbb{B}=\bigoplus_{\kappa\in \mathcal{T}_h} \mathcal{B}_{\kappa},\qquad
\mathcal{V}_{h}=\{{\bf v}_h\in \mathcal{C}^0(\overline{\Omega}); {\bf v}_{h|\kappa}\in \mathbb{P}_1 \; \forall \kappa\in \mathcal{T}_h,
\; {\bf v}_{h|\Gamma_0}=0, \mbox{ and }{\bf v}_{h}.\textbf{n}_{|\Gamma}=0\},\\\\
 \displaystyle {\bf V}_h=[\mathcal{V}_{h}+ \mathbb{B}]^2, \qquad \mathcal{W}_h=\{{\bf v}_{h|\Gamma}, \quad {\bf v}_h\in {\bf V}_h\},	\\\\
\displaystyle \mathbb{L}_h=\{q_h\in \mathcal{C}^0(\overline{\Omega}); q_{h|\kappa}\in \mathbb{P}_1\, \forall \kappa\in \mathcal{T}_h,\; \int_\Omega q_h=0\},		\\\\
\mathcal{Q}_h=\left\lbrace \displaystyle \mu_h \in \mathcal{W}_h, \, \int_{\Gamma}\mu_h \psi_h - \int_{\Gamma} g |\psi_h|\leq 0 \, \forall \psi_h\in \mathcal{W}_h\right\rbrace,\\\\
\mathcal{M}_h= \mathbb{L}_h\times \mathcal{W}_h, \qquad \Lambda_h=\mathbb{L}_h\times \mathcal{Q}_h.
\end{array}
\end{eqnarray*}

\section*{Remark}
$\mathcal{Q}_h$ is an external approximation of $\mathcal{Q}$, so the discretization is non-conforming and would weaken its convergence order.

\noindent Disicretizing (\ref{pb_mix_3cham_vect}) we obtain
%
\begin{eqnarray}
\label{pb_mix_3cham_vect_disc}
\left\{
\begin{array}{rcll}
\mbox{Find }({\bf u}_h,(p_h, \lambda_h)) \in {\bf V}_h\times \Lambda_h \mbox{such that:}\\\\
a({\bf u}_h,{\bf v}_h)+ b((p_h,\lambda_h),{\bf v}_h)	&=&	L({\bf v}_h)	& \forall {\bf v}_h\in {\bf V}_h,\\\\
\displaystyle 		b((q_h-p_h,\mu_h-\lambda_h),{\bf u}_h)	&\leq&	0		& \forall (q_h,\mu_h) \in \Lambda_h,
\end{array}
\right.
\end{eqnarray}
where $\Lambda_h$ is a closed convex of $\mathcal{M}_h$.

\noindent A sufficient condition for the existence and uniqness of the solution to problem (\ref{pb_mix_3cham_vect_disc}) is the {\it inf-sup}  condition \cite{brezzi78}.
%
\begin{lemma}
 \label{inf_sup_disc_prop}
These two propositions are equivalent:
\begin{itemize}
	\item [$\ast$]	There exists a constant $\beta>0$ independent of $h$  such that :
		\begin{eqnarray}
		\label{inf_sup_disc_vect}
		\begin{array}{lcll}
		\displaystyle \sup\limits_{{\bf v}_h\in{\bf V}_h}\frac{b((q_h,\mu_h), {\bf v}_h)}{||{\bf v}_h||_{1} }\geq\beta\left( ||q_h||_0 +||\mu_h||_{-\frac{1}{2}}\right) \qquad \forall (q_h,\mu_h)\in \mathcal{M}_h.
		\end{array}
		\end{eqnarray}
	\item[$\ast$]	There exists two constants $\beta_1>0$ and  $\beta_2>0$ independent of $h$  such that :
			\begin{eqnarray}
			\label{2cond_infsup}
			\left\{
			\begin{array}{l}
			\displaystyle \sup\limits_{{\bf v}_h\in {\bf V}_h}\frac{\left(q_h,div{\bf v}_h\right)}{ ||{\bf v}_h||_1}\geq \beta_1 ||q_h||_0 \qquad \forall\, {q_h\in \mathbb{L}_h}\\\\
			\mbox{and}\\
			\displaystyle \sup\limits_{{\bf v}_h\in {\bf V}_h}\frac{\left\langle \mu_h, {\bf v}_{th}\right \rangle}{ ||{\bf v}_h||_1}\geq \beta_2 ||{\mu}_h||_{-\frac{1}{2}} \qquad \forall\, {\mu_h\in \mathcal{W}_h}.
			\end{array}
			\right.
			\end{eqnarray}
\end{itemize}

%
\end{lemma}
%
\begin{proof}
 \label{proof_sup_sup}

\noindent To prove this result it suffices to show that:
\begin{equation}
\label{equiv_sup_sup}
  \displaystyle \sup\limits_{{\bf v}_h\in {\bf V}_h}  \left\lbrace \left(\mu_h, {\bf v}_{th}\right) \, + \, \left(q_h, div({\bf v}_h)\right)\right\rbrace = \sup\limits_{{\bf v}_h\in {\bf V}_h}  \left\lbrace \left(\mu_h, {\bf v}_{th}\right)\right\rbrace + \sup\limits_{{\bf v}_h\in {\bf V}_h}  \left\lbrace \left(q_h, div({\bf v}_h)\right)\right\rbrace
\end{equation}

\noindent Let us prove the non trivial direction :
\begin{equation}
\label{recip}
 \sup\limits_{{\bf v}_h\in {\bf V}_h}  \left\lbrace \left(\mu_h, {\bf v}_{th}\right)\right\rbrace + \sup\limits_{{\bf v}_h\in {\bf V}_h}  \left\lbrace \left(q_h, div({\bf v}_h)\right)\right\rbrace \leq \sup\limits_{{\bf v}_h\in {\bf V}_h}  \left\lbrace \left(\mu_h, {\bf v}_{th}\right) \, + \, \left(q_h, div({\bf v}_h)\right)\right\rbrace
\end{equation}

\noindent which is equivalent to
%
$$
\label{recip2}
 \left(\mu_h, {\bf u}_{th}\right)+ \left(q_h, div({\bf w}_h)\right) \leq \sup\limits_{{\bf v}_h\in {\bf V}_h}  \left\lbrace \left(\mu_h, {\bf v}_{th}\right) \, + \, \left(q_h, div({\bf v}_h)\right)\right\rbrace \qquad \, \forall ({\bf u}_h,{\bf w}_h) \in {\bf V}_h^2
$$
%

\vspace{0.5cm}
\noindent We suppose that (\ref{recip}) is not valid, ie $\exists \, \left(\theta_h,\omega_h\right)\in {\bf V}_h^2$ such that :
\begin{equation}
\label{absud1}
\displaystyle \sup\limits_{{\bf v}_h\in {\bf V}_h}  \left\lbrace \left(\mu_h, {\bf v}_{th}\right) \, + \, \left(q_h, div({\bf v}_h)\right)\right\rbrace < \left(\mu_h, \theta_h\right) + \left(q_h, div(\omega_h)\right)
\end{equation}
\begin{eqnarray*}
\label{absud2}
\displaystyle  (\ref{absud1})	&\Rightarrow &	\left(\mu_h, v_{th}\right) \, + \, \left(q_h, div(v_{th})\right)< \left(\mu_h, \theta_h\right) + \left(q_h, div(\omega_h)\right) \qquad \; \forall \, v_h\in {\bf V}_h \\\\
				&\Rightarrow &	\left(\mu_h, v_{th}-\theta_h\right) \, + \, \left(q_h, div(v_{h}-\omega_h)\right)< 0.
\end{eqnarray*}
%

\noindent Since we know that:
\begin{equation}
\label{inf_sup_sup}
 \forall \mu_h \in (L^2_h(\Gamma))^{d-1}\quad \gamma ||\mu_h|| \leq \sup\limits_{v_h\in {\bf V}_h} \frac{\left(\mu_h, v_{th}\right)}{||v_h||}
\; \Leftrightarrow \;
\left\{
 \begin{array}{l}
 \displaystyle \forall \mu_h \in (L^2_h(\Gamma))^{d-1} \; \exists \overline{\varphi}_h\,  \in {\bf V}_h \mbox{ such that }\\\\
 \left(\mu_h, \overline{\varphi}_h\right)\geq \gamma||\overline{\varphi}_h||\, \quad ||\mu_h||, \gamma >0
\end{array}
\right\}.
\end{equation}
and suppose that $\left(\mu_h, v_{th}-\theta_h\right)<0$ $\forall \, v_h\in {\bf V}_h$, wich remains valid for $v_h=\overline{\varphi}_h+\theta_h$ so that
$$
\left(\mu_h, \overline{\varphi}_h\right)<0,
$$
which contradicts (\ref{inf_sup_sup}).

\noindent The same reasoning can be applied to $\left(q_h, div(v_{h}-\omega_h)\right)< 0$ since $\left(\cdot , div(\cdot) \right)$ verifies similar {\it inf-sup} condition. This ends the proof.

\end{proof}
%
%
\begin{proposition}
\label{inf_sup_disc_prop2}
 There exists a constant $\beta>0$ independent of $h$  such that :
\begin{eqnarray}
\label{inf_sup_disc_vect2}
\begin{array}{lcll}
\displaystyle \sup\limits_{{\bf v}_h\in{\bf V}_h}\frac{b((q_h,\mu_h), {\bf v}_h)}{||{\bf v}_h||_{1} }\geq\beta\left( ||q_h||_0 +||\mu_h||_{-\frac{1}{2}}\right) \qquad \forall (q_h,\mu_h)\in \mathcal{M}_h
\end{array}
\end{eqnarray}
%
\end{proposition}
%
\begin{proof}
Recall that
$$
 \forall \left(\left(q_h,\mu_h\right),{\bf v}_h\right)\in \mathcal{M}_h \times {\bf V}_h  \qquad b\left(\left(q_h,\mu_h\right), {\bf v}_h\right)=\left(-q_h, div({\bf v}_h)\right)+\left\langle \mu_h, {\bf v}_{th}\right \rangle
$$
According to lemma \ref{inf_sup_disc_prop}, it suffices to show that :

\begin{equation}
\label{infsup_qh}
\displaystyle \exists\, \beta_1>0 \, \mbox{ tel que }\forall\, {q_h\in \mathbb{L}_h} \quad 
\sup\limits_{{\bf v}_h\in {\bf V}_h}\frac{\left(q_h,div{\bf v}_h\right)}{ ||{\bf v}_h||_1}\geq \beta_1 ||q_h||_0
\end{equation}
and
\begin{equation}
\label{infsup_muh}
\displaystyle \exists\, \beta_2>0 \, \mbox{ tel que }\forall\, {\mu_h\in \mathcal{W}_h} \quad 
\sup\limits_{{\bf v}_h\in {\bf V}_h}\frac{\left\langle \mu_h, {\bf v}_{th}\right \rangle}{ ||{\bf v}_h||_1}\geq \beta_2 ||{\mu}_h||_{-\frac{1}{2}}
\end{equation}
which are both (\ref{infsup_qh}) and (\ref{infsup_muh}) established in \cite{brezzi84} and \cite{faker03} respctively.
\end{proof}
%
Now we will derive error estimates for primal variable, being inspired by \cite{haslinger04}.

\begin{lemma}\cite{lhalouani99}
\label{lemme_intro}
Let $\left({\bf u}, p, \lambda \right)$ and $\left({\bf u}_h, p_h, \lambda_h \right)$ be solutions to (\ref{pb_mix_3cham_vect}), (\ref{pb_mix_3cham_vect_disc}) respectively. Then for any $\left({\bf v}_h, q_h, \mu_h \right)\in {\bf V}_h\times \Lambda_h$ it holds:
%
\begin{eqnarray}
\label{eq_lem_intro}
\begin{array}{rcl}
a({\bf u}-{\bf u}_h,{\bf u}-{\bf u}_h)	&\leq&	a({\bf u}-{\bf u}_h,{\bf u}-{\bf v}_h) 												+b((p,\lambda)-(q_h,\mu_h),{\bf u}_h-{\bf u})											+b((p,\lambda)-(p_h,\lambda_h),{\bf u}-{\bf v}_h)\\\\				
						&&+b((p,\lambda)-(q_h,\mu_h),{\bf u}) + b((p_h,\lambda_h)-(p,\lambda),{\bf u}).
\end{array}
\end{eqnarray}
%
\end{lemma}
%
\begin{proof}
 Let ${\bf v}_h$ be an element of ${\bf V}_h$. It follows that:
$$
a({\bf u}-{\bf u}_h,{\bf u}-{\bf u}_h)=a({\bf u}-{\bf u}_h,{\bf u}-{\bf v}_h)+a({\bf u}-{\bf u}_h,{\bf v}_h-{\bf u}_h).
$$
Using the first equations of (\ref{pb_mix_3cham_vect}) and of (\ref{pb_mix_3cham_vect_disc}), this gives :
%
\begin{eqnarray*}
\begin{array}{rcl}
a({\bf u}-{\bf u}_h,{\bf v}_h-{\bf u}_h)	&=&	a({\bf u},{\bf v}_h-{\bf u}_h)-a({\bf u}_h,{\bf v}_h-{\bf u}_h),\\\\				
						&=&	L({\bf v}_h-{\bf u}_h) -b((p,\lambda),{\bf v}_h-{\bf u}_h) - L({\bf v}_h-{\bf u}_h) + b((p_h,\lambda_h),{\bf v}_h-{\bf u}_h),\\\\
						&=&	b((p,\lambda),{\bf u}_h-{\bf v}_h) +b((p_h,\lambda_h),{\bf v}_h-{\bf u}_h).\\\\ 
\end{array}
\end{eqnarray*}
%
Then we deduce
$$
a({\bf u}-{\bf u}_h,{\bf u}-{\bf u}_h)=a({\bf u}-{\bf u}_h,{\bf u}-{\bf v}_h)+b((p,\lambda),{\bf u}_h-{\bf v}_h) +b((p_h,\lambda_h),{\bf v}_h-{\bf u}_h).
$$
Finally we have 

\begin{eqnarray*}
\begin{array}{rcl}
a({\bf u}-{\bf u}_h,{\bf u}-{\bf u}_h)	&=&	a({\bf u}-{\bf u}_h,{\bf u}-{\bf v}_h)+b((p,\lambda)-(q_h,\mu_h),{\bf u}_h-{\bf u})+b((p,\lambda)-(p_h,\lambda_h),{\bf u}-{\bf v}_h)\\\\				
					&&	+b((p,\lambda)-(q_h,\mu_h),{\bf u}) + b((p_h,\lambda_h)-(p,\lambda),{\bf u})\\\\
					&&	+ b((q_h,\mu_h)-(p_h,\lambda_h),{\bf u}_h).
\end{array}
\end{eqnarray*}
%

But according to (\ref{pb_mix_3cham_vect_disc}), $b((q_h,\mu_h)-(p_h,\lambda_h),{\bf u}_h) \leq 0$ for all $(q_h,\mu_h) \in \Lambda_h$. This ends the proof of the lemma.
%
%
%
%
\end{proof}
%

We now derive an upper bound of the terms involved in (\ref{eq_lem_intro}).
%
\begin{lemma}
\label{lem_estim_u2}
 Let $\left({\bf u}, p, \lambda \right)$ and $\left({\bf u}_h, p_h, \lambda_h \right)$ be solutions to (\ref{pb_mix_3cham_vect}), (\ref{pb_mix_3cham_vect_disc}) respectively. Suppose that ${\bf u}\in {\bf H}^2(\Omega)$ and $p\in H^1(\Omega)$. Then

%
\begin{equation}
\label{estim_u-uh2}
\displaystyle ||{\bf u}-{\bf u}_h||_1^2 \leq C({\bf u},p,g) \left( h||\lambda-\lambda_h||_{-\frac{1}{2}} + h^{\frac{3}{2}}\right)  
\end{equation}
where $C({\bf u},p,g)$ is a positive constant depending only on $||{\bf u}||_2$, $||p||_1$ and $||g||_{L^2(\Gamma)}$.
\end{lemma}
%
\begin{proof}
 Using Lemma \ref{lemme_intro}, we will show that there exists $({\bf v}_h, (q_h,\mu_h))\in {\bf V}_h\times \Lambda_h$ satisfying :
\begin{eqnarray}
\label{proof_lemm_a5}
\left\lbrace 
 \begin{array}{lcl}
\displaystyle a({\bf u}-{\bf u}_h, {\bf u}-{\bf v}_h)	&\leq&	C({\bf u}) h ||{\bf u}-{\bf u}_h||_1,\\\\
%
\displaystyle b((p,\lambda)-(q_h,\mu_h),{\bf u}_h-{\bf u})	&\leq&	C({\bf u},p) h ||{\bf u}-{\bf u}_h||_1,\\\\
%
\displaystyle b((p_h,\lambda_h)-(p,\lambda),{\bf u}-{\bf v}_h)		&\leq&	Ch \left\lbrace ||{\bf u}-{\bf u}_h||_1+ C(p)h + C({\bf u})h\right\rbrace\\\\
\displaystyle b((p,\lambda)-(q_h,\mu_h),{\bf u})		&\leq&	C({\bf u})^2 h^2\\\\
\displaystyle b((p_h,\lambda_h)-(p,\lambda),{\bf u})		&\leq&	C({\bf u})\left(  h ||\lambda-\lambda_h||_{-\frac{1}{2}}+C({\bf u}) h^{\frac{3}{2}}+C({\bf g}) h^{\frac{3}{2}}\right) \\\\
\end{array}
\right.
\end{eqnarray}

Before proving these estimates,  we first have to recall some useful results. Let $\mathcal{I}_h$, $\mathfrak{J}_h$ and $i_h$ be the Lagrange interpolation operators on ${\bf V}_h$, $\mathbb{L}_h$ and $\mathcal{W}_h$ respectively. From \cite{ciarlet80}, there exists a positive constant $C$ such that $\forall {\bf v}\in {\bf H}^2(\Omega)$, $\forall p\in \mathbb{L}_h$ and $\forall  \psi \in H^{\frac{3}{2}}(\Gamma)$:
\begin{equation}
\label{lag_interp}
 ||{\bf v}-\mathcal{I}_h{\bf v}||_1\leq C h ||{\bf v}||_2,  \qquad ||p-\mathfrak{J}_hp||_0\leq C h ||p||_1, \qquad  ||\psi-i_h\psi||_{0,\Gamma}\leq C h^{\frac{3}{2}} ||\psi||_{\frac{3}{2},\Gamma}.
\end{equation}
Let $\pi_h$ be the projection operator from $(L^2(\Gamma))^{d-1}$ on $\mathcal{W}_h$ defined by:
\begin{equation}
\label{proj_gamma}
\displaystyle \pi_h\psi\in \mathcal{W}_h, \qquad\qquad  \int_\Gamma\left(\pi_h\psi-\psi\right)\mu_h=0 \quad \forall \, \mu_h\in \mathcal{W}_h.
\end{equation}
It holds $\forall \tau \in [0,1] \mbox{ and } \forall \nu \in [0, \tau+\frac{1}{2}]$ one has:
\begin{equation}
\label{proj_approx}
\forall \psi\in {\bf H}^{\frac{1}{2}+\tau} \qquad\quad %
h^{-\frac{1}{2}}||\psi-\pi_h\psi||_{-\frac{1}{2},\Gamma}+h^{\nu}||\psi-\pi_h\psi||_{\nu,\Gamma}\leq C h^{\tau+\frac{1}{2}}.
\end{equation}
Let $\Pi_h$ be the projection operator from  $L^2(\Omega)$ on $\mathbb{L}_h$ defined by:
\begin{equation}
\label{proj_l2}
\displaystyle \Pi_hq\in \mathbb{L}_h, \qquad\qquad  \int_\Omega\left(\Pi_hq-q\right)s_h=0 \quad \forall \, s_h\in \mathbb{L}_h.
\end{equation}
Finally, let us note the trace theorem implies that
\begin{equation}
\label{trace_estim}
 ||\lambda||_{\frac{1}{2},\Gamma}\leq C||{\bf u}||_2
\end{equation}

\noindent \textit{\textbf{(i)}} The first term is evaluted by using the continuity of $a(\cdot, \cdot)$ and the property (\ref{lag_interp})
\begin{eqnarray*}
 \begin{array}{rcl}
a({\bf u}-{\bf u}_h,{\bf u}-{\bf v}_h)	&\leq&	C ||{\bf u}-{\bf u}_h||_1\, ||{\bf u}-{\bf v}_h||_1 \quad \forall {\bf v}_h\in {\bf V}_h,\\\\
					&\leq&	C \inf\limits_{{\bf v}_h\in {\bf V}_h}\left\lbrace ||{\bf u}-{\bf v}_h||_1\right\rbrace  ||{\bf u}-{\bf u}_h||_1,\\\\
					&\leq& C ||{\bf u}-\mathcal{I}_h{\bf u}||_1 ||{\bf u}-{\bf u}_h||_1,\\\\
					&\leq& C({\bf u}) h ||{\bf u}-{\bf u}_h||_1.
\end{array}
\end{eqnarray*}

\noindent \textit{\textbf{(ii)}} Using (\ref{lag_interp}) and (\ref{trace_estim}) we have
\begin{eqnarray}
\label{terme2_u-uh}
 \begin{array}{rcl}
b((p,\lambda)-(q_h,\mu_h),{\bf u}_h-{\bf u})	&=&	-(p-q_h, div({\bf u}_h-{\bf u})) + \left\langle \lambda-\mu_h,{\bf u}_{ht}-{\bf u}_{t} \right\rangle, \\\\
						&\leq&	C\left\lbrace ||p-q_h||_{0}+||\lambda-\mu_h||_{-\frac{1}{2}}\right\rbrace ||{\bf u}-{\bf u}_h||_1 \quad \forall (q_h,\mu_h)\in \Lambda_h,\\\\
						&\leq&	C\left\lbrace \inf\limits_{q_h\in \mathbb{L}_h} ||p-q_h||_{0}+\inf\limits_{\mu_h\in \mathcal{Q}_h}||\lambda-\mu_h||_{-\frac{1}{2}}\right\rbrace ||{\bf u}-{\bf u}_h||_1,\\\\
						&\leq& C({\bf u},p)h ||{\bf u}-{\bf u}_h||_1.
\end{array}
\end{eqnarray}

\noindent \textit{\textbf{(iii)}} Further
%
\begin{eqnarray}
\label{terme3_u-uh}
 \begin{array}{rcl}
b((p,\lambda)-(p_h,\lambda_h),{\bf u}-{\bf v}_h)	&=&	-(p-p_h, div({\bf u}-{\bf v}_h)) + \left\langle \lambda-\lambda_h,{\bf u}_t-{\bf v}_{ht} \right\rangle, \\\\
						&\leq&	C\left\lbrace ||p-p_h||_{0}+||\lambda-\lambda_h||_{-\frac{1}{2}}\right\rbrace ||{\bf u}-{\bf v}_h||_1 \quad \forall {\bf v}_h\in{\bf V}_h,\\\\
%
%
					&\leq& C \left\lbrace ||p-q_h||_0+||\lambda-\mu_h||_{-\frac{1}{2},\Gamma}\right.\\\\
					&&	\left. +||q_h-p_h||_0+||\lambda_h-\mu_h||_{-\frac{1}{2},\Gamma} \right\rbrace ||{\bf u}-{\bf v}_h||_1.\\\\
\end{array}
\end{eqnarray}
%

Since $(p_h-q_h, \lambda_h-\mu_h)\in \mathcal{M}_h$ then it follows from the discrete {\it inf-sup} condition (\ref{inf_sup_disc_vect2}):
\begin{eqnarray*}
 \begin{array}{lcl}
 \beta \left(||p_h-q_h||_0+||\lambda_h-\mu_h||_{-\frac{1}{2},\Gamma}\right)	&\leq&	\displaystyle\sup\limits_{{\bf v}_h\in {\bf V}_h} \frac{b\left(\left(p_h-q_h,\lambda_h-\mu_h\right),{\bf v}_h\right)}{||{\bf v}_h||_1},\\\\
				&\leq&	\displaystyle\sup\limits_{{\bf v}_h\in {\bf V}_h}\left\lbrace  \frac{b\left(\left(p_h-p,\lambda_h-\lambda\right),{\bf v}_h\right)}{||{\bf v}_h||_1} +  \frac{b\left(\left(p-q_h,\lambda-\mu_h\right),{\bf v}_h\right)}{||{\bf v}_h||_1}\right\rbrace, \\\\
				&\leq&	\displaystyle\sup\limits_{{\bf v}_h\in {\bf V}_h}\frac{a({\bf u}-{\bf u}_h,{\bf v}_h)}{||{\bf v}_h||_1}+ ||p-q_h||_0+C||\lambda-\mu_h||_{-\frac{1}{2},\Gamma}.
\end{array}
\end{eqnarray*}
Hence : $\forall \left(q_h,\mu_h\right)\in \mathcal{M}_h$
\begin{equation}
\label{infsup_resul2}
 ||p_h-q_h||_0+||\lambda_h-\mu_h||_{-\frac{1}{2},\Gamma}\leq C \left(||{\bf u}-{\bf u}_h||_1+||p-q_h||_0+||\lambda-\mu_h||_{-\frac{1}{2},\Gamma}\right).
\end{equation}
%
Then combining the last inequality of (\ref{terme2_u-uh}) and (\ref{infsup_resul2}) and using property (\ref{lag_interp}) we obtain
\begin{eqnarray*}
 \begin{array}{lcl}
 b((p,\lambda)-(p_h,\lambda_h),{\bf u}-{\bf v}_h)	&\leq&	C \left\lbrace ||{\bf u}-{\bf u}_h||_1+ \inf\limits_{q_h\in \mathbb{L}_h}||p-q_h||_0\right.\\
							&&	\left.+ \inf\limits_{\mu_h\in \mathcal{Q}_h}||\lambda-\mu_h||_{-\frac{1}{2}}\right\rbrace\inf\limits_{{\bf v}_h\in {\bf V}_h}||{\bf u}-{\bf v}_h||_1,\\\\
							&\leq&	Ch \left\lbrace ||{\bf u}-{\bf u}_h||_1+ C(p)h + C({\bf u})h\right\rbrace
\end{array}
\end{eqnarray*}
%

\noindent \textit{\textbf{(iv)}} To estimate this term we invoke the definition of the $L^2$-projection operator:
%
\begin{eqnarray}
\label{terme4_u-uh}
 \begin{array}{rcl}
b((p,\lambda)-(q_h,\mu_h),{\bf u})	&=&\displaystyle	-(p-q_h, div{\bf u}) + \int_{\Gamma}(\lambda-\mu_h){\bf u}_t d\Gamma \quad \forall \mu_h\in \mathcal{Q}_h, \\\\
					&=&\displaystyle	\int_{\Gamma}(\lambda-\pi_h \lambda){\bf u}_t d\Gamma \quad \mbox{ for } \mu_h=\pi_h\lambda, \\\\
					&=&\displaystyle	\int_{\Gamma}(\lambda-\pi_h \lambda){\bf u}_t d\Gamma - \int_{\Gamma}(\lambda-\pi_h \lambda)\pi_h {\bf u}_t d\Gamma, \\\\	
					&=&\displaystyle	\int_{\Gamma}(\lambda-\pi_h \lambda)({\bf u}_t-\pi_h {\bf u}_t) d\Gamma, \\\\
					&\leq&||\lambda-\pi_h \lambda||_{L^2(\Gamma)} ||{\bf u}_t-\pi_h {\bf u}_t||_{L^2(\Gamma)},\\\\

					&\leq& C({u})^2h^2.
\end{array}
\end{eqnarray}
%

\textit{\textbf{(v)}} Now we shall estimate the fifth term of (\ref{proof_lemm_a5}) using (\ref{equiv_tresca}) and the definition of $\mathcal{Q}_h$
\begin{eqnarray*}
 \begin{array}{rcl}
b((p_h,\lambda_h)-(p,\lambda),{\bf u})	&=&\displaystyle	-(p_h-p, div{\bf u}) + \int_{\Gamma}(\lambda_h-\lambda){\bf u}_t d\Gamma, \\\\
					&=&\displaystyle	\int_{\Gamma}(\lambda_h-\lambda){\bf u}_t d\Gamma, \\\\
					&=&\displaystyle	\int_\Gamma \left(\lambda_h-\lambda\right)\left({\bf u}_t-i_h({\bf u}_t)\right) +  \int_\Gamma \left(\lambda_h-\lambda\right)i_h({\bf u}_t)   +   \int_{\Gamma}\left(\lambda {\bf u}_{t}-g|{\bf u}_{t}|\right),\\\\
				&\leq&	\displaystyle \int_\Gamma \left(\lambda_h-\lambda\right)\left({\bf u}_t-i_h({\bf u}_t)\right) + \int_\Gamma g\left(|i_h({\bf u}_t)|-|{\bf u}_t|\right)   + \int_\Gamma \lambda \left({\bf u}_t-i_h({\bf u}_t)\right),\\\\
				&\leq&	\displaystyle \int_\Gamma \left(\lambda_h-\lambda\right) \left({\bf u}_t-i_h({\bf u}_t)\right) + \int_\Gamma g|i_h({\bf u}_t)-{\bf u}_t|   + \int_\Gamma \lambda\left({\bf u}_t-i_h({\bf u}_t)\right),\\\\
				&\leq&	||\lambda_h-\lambda||_{-\frac{1}{2},\Gamma} ||{\bf u}_t-i_h({\bf u}_t)||_{\frac{1}{2},\Gamma}  +  ||g||_{0,\Gamma}||{\bf u}_t-i_h({\bf u}_t)||_{0,\Gamma}\\
				& &+   ||\lambda||_{0,\Gamma}||{\bf u}_t-i_h({\bf u}_t)||_{0,\Gamma},\\\\
				&\leq&	C({\bf u})h||\lambda-\lambda_h||_{-\frac{1}{2},\Gamma}  +  C({\bf u})^2h^{\frac{3}{2}}+ C(g)C({\bf u})h^{\frac{3}{2}},\\\\
				&\leq&	C({\bf u},g) \left\lbrace h||\lambda-\lambda_h||_{-\frac{1}{2},\Gamma} +h^{\frac{3}{2}}\right\rbrace.
\end{array}
\end{eqnarray*}
Assembling the estimates \textit{\textbf{(i)}}-\textit{\textbf{(v)}} in the Lemma \ref{lem_estim_u2} and using the V-ellipticity of the bilinear form $a(\cdot, \cdot)$, we finally arrive at the following estimate
$$
||{\bf u}-{\bf u}_h||_1^2 \leq C({\bf u},p) \left( h||\lambda-\lambda_h||_{-\frac{1}{2}}+h||{\bf u}-{\bf u}_h||\right) + C({\bf u},g)h^{\frac{3}{2}},
$$
then using the Young inequality we can write for every constant $\beta>0$
$$
C({\bf u},p)h||{\bf u}-{\bf u}_h||_1\leq C({\bf u},p)\left(\beta h^2+ \frac{1}{\beta}||{\bf u}-{\bf u}_h||_1^2\right).
$$
Taking $\beta$ such that $\displaystyle\frac{C({\bf u},p)}{\beta}<1$ then leads to the desired result.
\end{proof}
%
\begin{lemma}
\label{lem_estim_p_lambda}
 Let $\left({\bf u}, p, \lambda \right)$ and $\left({\bf u}_h, p_h, \lambda_h \right)$ be solutions to (\ref{pb_mix_3cham_vect}), (\ref{pb_mix_3cham_vect_disc}) respectively. Suppose that ${\bf u}\in {\bf H}^2(\Omega)$ and $p\in H^1(\Omega)$. Then
%
\begin{equation}
\label{estim_p-ph__lam_lamh}
\displaystyle ||p-p_h||_0 +||\lambda-\lambda_h||_{-\frac{1}{2}} \leq C({\bf u},p)\left\lbrace h+||{\bf u}-{\bf u}_h||_1\right\rbrace, 
\end{equation}
where $C({\bf u},p)$ is a positive constant depending only on $||{\bf u}||_2$ and $||p||_1$.
\end{lemma}
%
%
\begin{proof}
Using  (\ref{terme3_u-uh}) and (\ref{infsup_resul2}) we get the disired result.
\end{proof}
%
\begin{theorem}
\label{lem_estim_u_p_lam}
Let $\left({\bf u}, p, \lambda \right)$ and $\left({\bf u}_h, p_h, \lambda_h \right)$ be solutions to (\ref{pb_mix_3cham_vect}), (\ref{pb_mix_3cham_vect_disc}) respectively. Suppose that ${\bf u}\in {\bf H}^2(\Omega)$ and $p\in H^1(\Omega)$. Then
$$
\displaystyle||{\bf u}-{\bf u}_h||_1+||p-p_h||_0 +||\lambda-\lambda_h||_{-\frac{1}{2}}\leq C({\bf u},p,g)h^{\frac{3}{4}},
$$ 
where $C({\bf u},g)$ is a where $C({\bf u},p,g)$ is a positive constant depending only on $||{\bf u}||_2$, $||p||_1$ and $||g||_{L^2(\Gamma)}$.
\end{theorem}
%
%
\begin{proof}
By assembling (\ref{estim_u-uh2}) and (\ref{estim_p-ph__lam_lamh}) we can write:
\begin{eqnarray*}
\begin{array}{rcl}
\displaystyle||{\bf u}-{\bf u}_h||_1^2	&\leq&	C({\bf u},p,g)\left\lbrace h ||\lambda-\lambda_h||_{-\frac{1}{2}}+h^{\frac{3}{2}}\right\rbrace \\\\
					&\leq&	C({\bf u},p,g)h \left\lbrace h+||{\bf u}-{\bf u}_h||_1\right\rbrace +C({\bf u},p,g) h^{\frac{3}{2}}\\\\
					&\leq&	C({\bf u},p,g)h^2+C({\bf u},p,g)h||{\bf u}-{\bf u}_h||_1 + C({\bf u},p,g) h^{\frac{3}{2}}\\\\
\end{array}
\end{eqnarray*}
then using Young's inequality we can  easily write:
$$
\displaystyle||{\bf u}-{\bf u}_h||_1 \leq C({\bf u},p,g)h^{\frac{3}{4}}
$$
so that (\ref{estim_p-ph__lam_lamh}) becomes
$$
||p-p_h||_0 +||\lambda-\lambda_h||_{-\frac{1}{2}}\leq C({\bf u},p,g)h^{\frac{3}{4}}
$$
wich leads to the desired result.

\end{proof}
%
\section{Numerical simulations}

We briefly describe the numerical resolution of the 2D Stokes problem with boundary conditions of Tresca friction type. For this aim, the augmented lagrangian method \cite{fortin_glow} will be used.

The minimization problem (\ref{pb_minimi}) is replaced by :
%
\begin{eqnarray}
\label{pb_mini_lagr_aug}
\left\{
\begin{array}{l}
\mbox{Find } ({\bf u},\Phi) \in \Pi \mbox{ such that :}\\\\
\displaystyle \Sigma({\bf u},\Phi) \leq \Sigma({\bf v},\varphi) \, \forall \, ({\bf v},\varphi)\in \Pi,
\end{array}
\right.
\end{eqnarray}
where 
$$\Pi=\{({\bf v},\varphi)\in {\bf V}_{div}(\Omega)\times L^2(\Gamma) \mbox{ such that }\varphi={\bf v}_t\},$$
and $\Sigma$ the lagrangian is defined on $\Pi$ by:
$$\displaystyle \forall (\varphi,{\bf v})\in \Pi \qquad \Sigma({\bf v},\varphi)=\frac{1}{2}a({\bf v},{\bf v})-L({\bf v})+j(\varphi).$$
Then, the following saddle-point problem is derived
%
\begin{eqnarray}
\label{pt_selle_aug}
\left\{
\begin{array}{l}
\mbox{Find }({\bf u},\Phi, \lambda) \in \Pi\times L^2(\Gamma)\mbox{ such that:}\\\\
\mathcal{L}_r({\bf u},\varphi,\mu) \leq \mathcal{L}_r({\bf u},\Phi,\lambda) \leq \mathcal{L}_r({\bf v},\Phi,\lambda) \quad  \forall ({\bf v},\varphi,\mu)\in {\bf V}_{div}(\Omega)\times (L^2(\Gamma))^2,
\end{array}
\right.
\end{eqnarray}
%
where
\begin{equation}
\label{fct_lagra_aug}
\displaystyle \mathcal{L}_r({\bf v},\varphi, \mu)=\Sigma({\bf v},\varphi)+\int_\Gamma ({\bf v}_t-\varphi)\mu+\frac{r}{2}||\varphi-{\bf v}_t||^2_{0,\Gamma},
\end{equation}
and we use bloc relaxation Uzawa algorithm, or ALG2 as mentionned in \cite{fortin_glow}, to solve (\ref{pt_selle_aug}). This leads to the following algorithm:
\begin{center}
\begin{boxedminipage}{14cm}

\vspace{0.25cm}
	{\tt
\begin{enumerate}
	\item Initialisation: $\Phi^{-1}$, $\lambda^0$ et $r>0$ fixed.
	\item Repeat until convergence
\begin{eqnarray}
\left\{
\begin{array}{rlll}
 \mbox{Find }\quad {\bf u}^k\in{\bf V} \mbox{ \footnotesize such:}& \forall {\bf v} \in {\bf V}\\\\
\displaystyle a({\bf u}^k,{\bf v})+r ({\bf u}^k_t,{\bf v}_t)_{\Gamma}	=&	L({\bf v})+(r\Phi^{k-1}-\lambda^k,{\bf v}_t)_{\Gamma}\,\\\\ 
div({\bf u}^k)								=&	0
\end{array}
\right.
\end{eqnarray}

---------
\begin{eqnarray}
\Phi^k=
\left\{
\begin{array}{rcll}
\displaystyle \frac{||\lambda^k+r{\bf u}^k_t||-g}{r||\lambda^k+r{\bf u}^k_t||} (\lambda^k+r{\bf u}^k_t) &\mbox{ if }& ||\lambda^k+r{\bf u}^k_t||\geq g\\\\
%
0 &\mbox{unless}
\end{array}
\right.
\end{eqnarray}

---------
\begin{eqnarray}
\begin{array}{lcll}
 \lambda^{k+1}=\lambda^k+\rho_k({\bf u}_t^k-\Phi^{k})
\end{array}
\end{eqnarray}

---------
	\item[]
	\item $\displaystyle \frac{||({\bf u}^k,\Phi^{k})-({\bf u}^{k-1},\Phi^{k-1})||}{||({\bf u}^k,\Phi^{k})||}<\varepsilon$ $\Rightarrow$ End.
\end{enumerate}
}

\vspace{0.25cm}
\end{boxedminipage}
\end{center}
\section*{Remarks}
\begin{itemize}
	\item[$\ast$] It's recommanded in \cite{fortin_glow,glow_tallec} to choose $\rho_k=\rho=r$ to ensure the convergnece of the above algorithm;
	\item[$\ast$] A second issue is how to choose $r$? Numerical tests show that there is an optimal value $r_{opt}$ for which convergence is the fastest. Unfortunately, this result still unprooved.
\end{itemize}
\subsection{Numerical Tests}
A no-slip 2D Stokes solver \cite{jonas_code} is used and Tresca friction boundary conditions were implemented on. $\Omega$ is the square $[0,0.1]^2$, the fluid can slip on $\Gamma=\Gamma_{upper}\cup \Gamma_{lower}=[0,0.1]\times\{0.1\}\cup [0,0.1]\times\{0\}$, the viscosity is taken equal to 0.1 and $10^{-6}$ is choosen as a stopping criterion.
\subsubsection{Test 1:}
If the threshold is never beeing reached then there is no-slip on all parts of the boundary $\partial \Omega$ wich is the case if the solution $({\bf u}, p)$ is that of the Stokes problem with homogeneous Dirichlet boundary conditions.
Logically, the value of $g$ has no effect on the solution of such problem.

\noindent The volume data ${\bf f}$ is adjusted so that the exact solution will be :
%
\begin{eqnarray*}
 \begin{array}{rcl}
	u_1(x,y)	&=&	-cos(20\pi x)sin(20\pi y)+sin(20\pi y)\\
	u_2(x,y)	&=&	-sin(20\pi x)cos(20\pi y)-sin(20\pi y)\\
	p(x,y)		&=&	20\pi(cos(20\pi y)-cos(20\pi x))
 \end{array}
\end{eqnarray*}
%

\noindent As shown in table \ref{tab:maill_param}, error decreases as we consider smaller mesh size.
{\footnotesize
\begin{table}[H]
\begin{center}
\begin{tabular}{|ccc||ccc|}
\hline
$1/h$ &$np$&$nt$& $||{\bf u}-{\bf u}_h||_0$ & $||{\bf u}-{\bf u}_h||_1$ & $||p-p_h||_0$ \\
\hline\hline
700	&	8522	&	16762	&	1.245e-04	&	1.365e-01	&	4.253e-03\\
900	&	14038	&	27714	&	7.446e-05	&	1.056e-01	&	2.992e-03 \\
1100	&	20880	&	41318	&	4.969e-05	&	8.656e-02	&	2.211e-03\\
1300	&	29506 	& 	58490 	& 	3.534e-05 	& 	7.305e-02 	&	1.700e-03\\
1500 	& 	39103 	& 	77604 	& 	2.647e-05 	& 	6.304e-02 	& 	1.451e-03\\
1600 	& 	44756 	& 	88870 	& 	2.306e-05 	& 	5.889e-02 	& 	1.317e-03\\
1700 	& 	50228 	& 	99774 	& 	2.062e-05 	& 	5.554e-02 	& 	1.262e-03\\
1800 	& 	56385 	& 	112048 	& 	1.837e-05 	& 	5.259e-02 	& 	1.170e-03\\
2000 	& 	69068 	& 	137334 	& 	1.514e-05 	& 	4.762e-02 	& 	9.956e-04\\
3000 	&	155610	& 	310018	& 	6.650e-06  	& 	3.159e-02 	& 	6.181e-04\\
\hline
\end{tabular}
\end{center}
\caption{$h$: mesh size, $np$: number of noeuds, $nt$: number of triangles}
\label{tab:maill_param}
\end{table}
}
%
{\footnotesize
\begin{table}[H]
\begin{center}
\begin{tabular}{|c|c|c|c|c|}
\hline
$g$ 	&$||{\bf u}-{\bf u}_h||_0$	& $||{\bf u}-{\bf u}_h||_1$	& $||p-p_h||_0$	& $n_{it}$ \\
\hline\hline
0	&	3.0405e-03		& 9.1986e-02			& 7.4967e-02 	& 26\\
0.015	&	3.0273e-03		& 9.1623e-02			& 7.7007e-02	& 131\\
10	&	3.0251e-03		& 9.1596e-02			& 7.7089e-02	& 135\\
40	&	3.0251e-03		& 9.1596e-02			& 7.7089e-02	& 135\\
\hline
\end{tabular}
\end{center}
\caption{ Effect of $g$ on the approximate solution. $r=10$, $n_{it}$:number of iteration to convergence}
\label{tab:g_effect}
\end{table}
}
\subsubsection{Test 2:}
%
We set $g=0.015$ wich is consistent with experimental values, see \cite{hatzik91} and \cite{leger03}, and we enforce parabolic profil on both $\Gamma_{left}$ and $\Gamma_{right}$:
\begin{eqnarray*}
{\bf u}_l={\bf u}_r=
\left[ 
 \begin{array}{lllll}
  y(1-y)\\
-y(1-y)
 \end{array}
\right] 
\end{eqnarray*}
where ${\bf u}_l={\bf u}|_{\Gamma_{left}}$ and ${\bf u}_r={\bf u}|_{\Gamma_{right}}$.

\noindent We choose this profile to enforce shear stress near the solid wall to reach the threshold without considering a complicated domain geometry. We can easily notice that fluid slips on some regions of $\partial \Omega$ and adheres the other regions, see figures (\ref{ligne_courant}, \ref{zoom_adh_gliss}).

{\footnotesize
\begin{figure}[H]
	\centering{\includegraphics[scale=0.30]{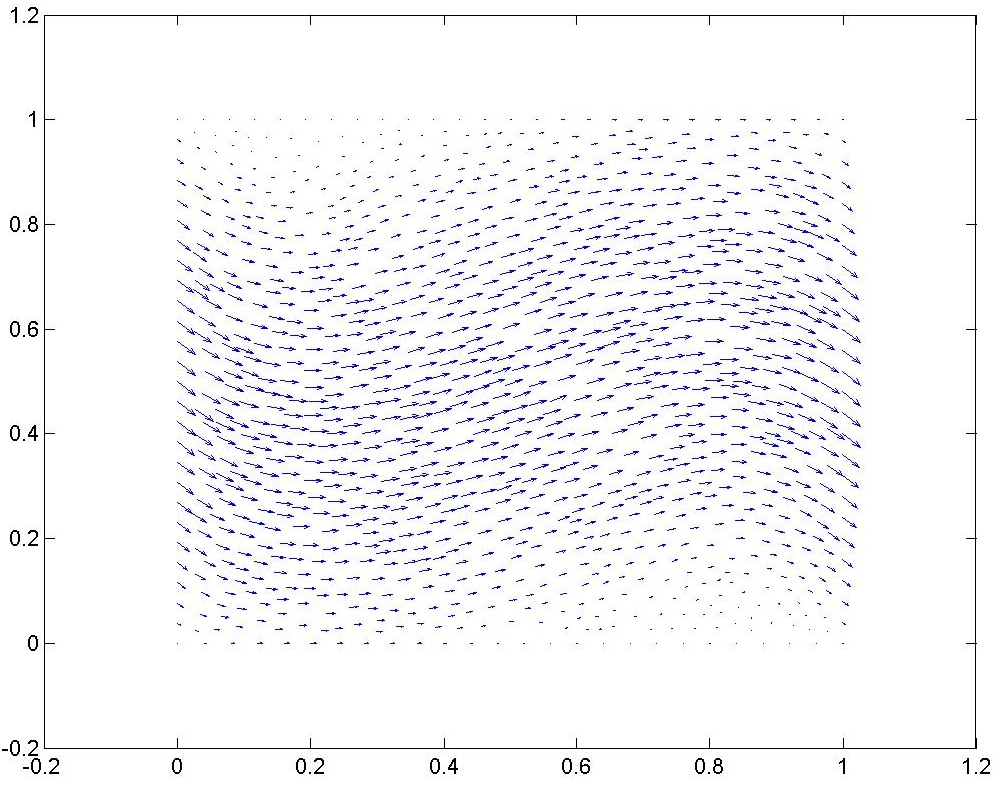}}
	\caption{{Fluid flow with boundary condition of Tresca friction type}}
	\label{ligne_courant}
\end{figure}

\begin{figure}[H]
\hspace*{\fill}
	\centering{\includegraphics[scale=0.15]{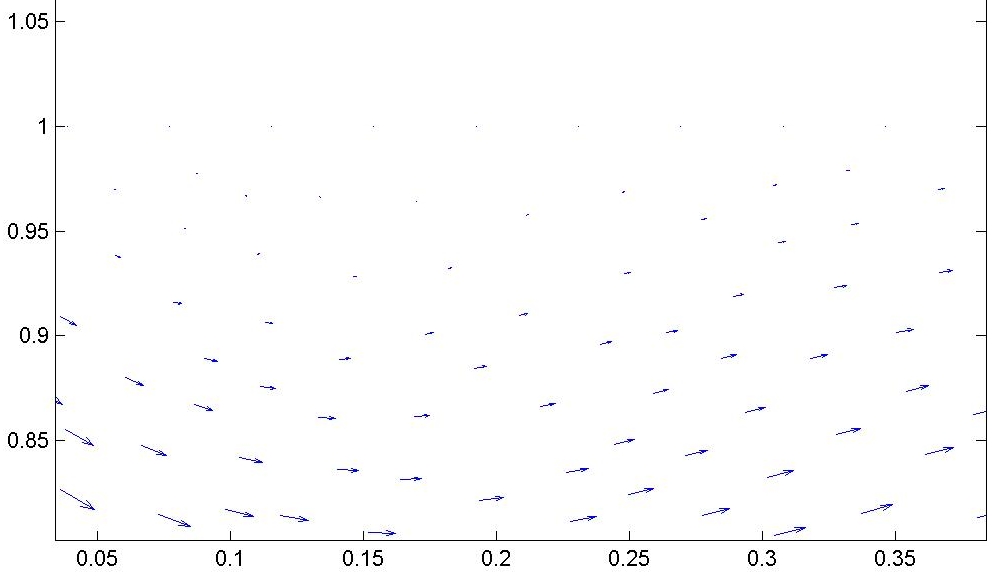}}
\hspace*{\fill}
		\centering{\includegraphics[scale=0.15]{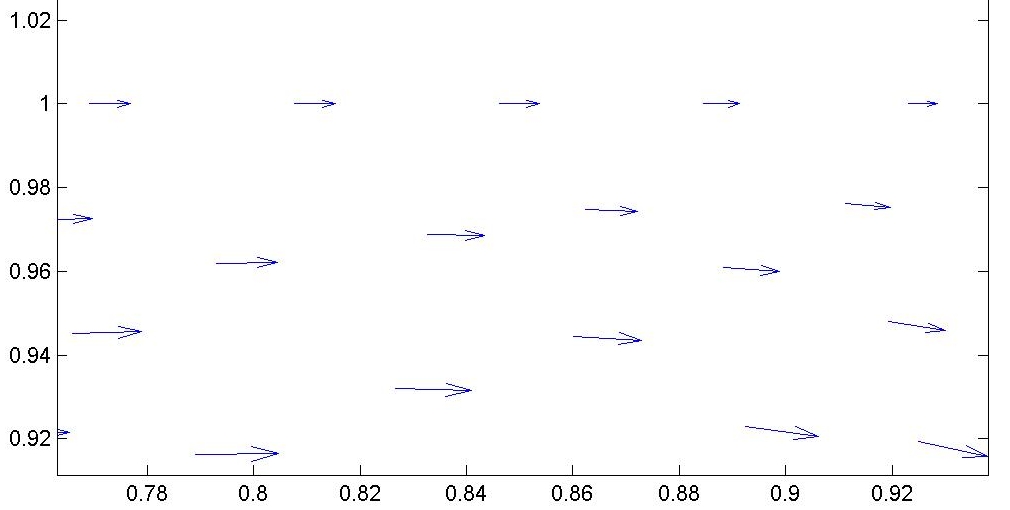}}
\hspace*{\fill}
		\caption{{Zoom of snon-slip and slip zones}}
\label{zoom_adh_gliss}
\end{figure}
}

Since an explicit solution to such a problem is not available, we calculate the discrete solution with sufficiently refined mesh, $h=\frac{1}{2000}$, which is taken as the reference solution; next we compute ${\bf u}_h$, the approximate solution, for different mesh sizes $h$ and we compare them to the reference solution.

\begin{table}[H]
\begin{center}
\begin{tabular}{|c|cc||cc||cc|}
\hline
$h$ & $||{\bf u}-{\bf u}_h||_0$ & $\alpha_0$ & $||{\bf u}-{\bf u}_h||_1$ & $\alpha_1$ & $||p-p_h||_0$ & $\alpha_p$\\
\hline
1.4286e-03   &3.5378e-04   &1.213   &8.8945e-03   &0.720   &2.0119e-01   &0.244\\
1.2500e-03   &2.7779e-04   &1.225   &7.6057e-03   &0.729   &1.7985e-01   &0.256\\
1.1111e-03   &2.2973e-04   &1.231   &6.8531e-03   &0.732   &1.6176e-01   &0.267\\
1.e-03       &1.8160e-04   &1.247   &5.9045e-03   &0.742   &1.4513e-01   &0.279\\
9.0909e-04   &1.5163e-04   &1.255   &5.3128e-03   &0.747   &1.3040e-01   &0.290\\
8.3333e-04   &1.2746e-04   &1.264   &5.0581e-03   &0.745   &1.1660e-01   &0.303\\
7.6923e-04   &1.0930e-04   &1.272   &4.6551e-03   &0.748   &1.0512e-01   &0.314\\
7.1429e-04   &9.5356e-05   &1.278   &4.3974e-03   &0.749   &6.0087e-02   &0.388\\
6.6667e-04   &8.5856e-05   &1.280   &4.6691e-03   &0.733   &8.9121e-02   &0.330\\
6.2500e-04   &7.3662e-05   &1.289   &3.7694e-03   &0.756   &3.9288e-02   &0.438\\
5.8824e-04   &6.5507e-05   &1.295   &4.1820e-03   &0.736   &5.8219e-02   &0.382\\
5.5556e-04   &5.8187e-05   &1.301   &3.6802e-03   &0.747   &4.9130e-02   &0.402\\
\hline
\end{tabular}
\end{center}
\caption{Convergence rates with respect to $h$}
\label{tab:erreur}
\end{table}

Table \ref{tab:erreur} provides the variation of $||{\bf u}-{\bf u}_h||_0$, $||{\bf u}-{\bf u}_h||_1$ and $||p-p_h||_0$ with respect to the mesh size respectively. The first remark one can make is the rate convergence of $H^1$-norm of error on ${\bf u}$ is equal to $\displaystyle \frac{3}{4}$ which is in agreement with theoretical result. The second one is that in spite of concidering very small mesh size, $\displaystyle h=\frac{1}{1800}$, we cannot conclude about rate convergence of ${\bf u}$ and $p$ error $L^2$-norms.
%
\section{Conclusion}
%
A three field mixed formulation of the stokes problem with Tresca boundary condition has been introduced and studied. The convergence analysis and a priori error estimates of the discrete corresponding problem have been established. In particular, we show an optimal error estimate of order $h^{\frac{3}{4}}$ for the velocity when it is approximated by classical {\it P1 bubble} finite element. A numerical realisation of a model example have been proposed wich confirms the theoritical result.
%

\end{document}